\documentclass{amsart}
\usepackage{graphicx}
\usepackage{amssymb,amscd,amsthm,amsxtra}
\usepackage{latexsym}
\usepackage{epsfig}
\usepackage{mathtools}
\usepackage{esint}
\usepackage{color}

\usepackage{cfr-lm}
\usepackage[osf]{libertine}
\usepackage{hyperref}
\hypersetup{colorlinks,breaklinks,
            linkcolor=[rgb]{0,0,0},
            citecolor=[rgb]{0,0,0},
            urlcolor=[rgb]{0,0,0}}

\newtheorem{thm}{Theorem}[section]
\newtheorem{cor}[thm]{Corollary}
\newtheorem{lem}[thm]{Lemma}

\theoremstyle{definition}
\newtheorem{defn}[thm]{Definition}
\theoremstyle{remark}

\numberwithin{equation}{section}
\newcommand{\be}{\begin{equation}}
\newcommand{\ee}{\end{equation}}

\newcommand{\abs}[1]{\left\vert#1\right\vert}

\newcommand{\R}{\mathbb R}

\newcommand{\eps}{\varepsilon}

\newcommand{\p}{\partial}

\newcommand{\comment}[1]{}
\newcommand\ddfrac[2]{\frac{\displaystyle #1}{\displaystyle #2}}
\newcommand{\mres}{\mathbin{\vrule height 1.6ex depth 0pt width
0.13ex\vrule height 0.13ex depth 0pt width 1.3ex}}

\begin{document}

\title[Improvement of flatness for vector valued free boundary problems]{Improvement of flatness for vector valued free boundary problems}
\author{D. De Silva}
\address{Daniela De Silva  \newline \indent Department of Mathematics  \newline \indent  Barnard College, Columbia University,  \newline \indent  New York, NY 10027}
\email{\tt  desilva@math.columbia.edu}
\author[G. Tortone]{G. Tortone}\thanks{}
\address{Giorgio Tortone \newline \indent Dipartimento di Matematica
	\newline\indent
	Alma Mater Studiorum Universit\`a di Bologna
	\newline\indent
	 Piazza di Porta San Donato 5, 40126 Bologna}
\email{giorgio.tortone@unibo.it}

\thanks{G. Tortone is partially supported by the ERC Advanced Grant 2013 n.\ 339958 Complex Patterns for Strongly Interacting Dynamical Systems - COMPAT, held by Susanna Terracini.}
\keywords{One-phase free boundary problem; Harnack Inequality; vectorial problem; viscosity solution; improvement of flatness}

\date{\today}%
\begin{abstract}
For a vectorial Bernoulli-type free boundary problem, with no sign assumption on the components, we prove that flatness of the free boundary
implies $C^{1,\alpha}$ regularity, as well-known in the scalar case \cite{AC,C2}.\\
While in \cite{MTV2} the same result is obtained for minimizing solutions by using a reduction to the scalar problem, and the NTA structure of the regular part of the free boundary, our result  uses directly a viscosity approach on the vectorial problem, in the spirit of \cite{D}.\\
We plan to use the approach developed here in vectorial free boundary problems involving a fractional Laplacian, as those treated in the scalar case in \cite{DR, DSS}.
\end{abstract}
\maketitle

\section{Introduction}

This note is concerned with the vector valued one-phase free boundary problem,
\be \begin{cases} \label{VOP}
\Delta U=0 & \text{in $\Omega(U):= \Omega \cap \{|U| > 0\}$;}\\
 |\nabla |U||=1 & \text{on $F(U):= \Omega \cap \p \Omega(U).$}
\end{cases}\ee
Here $U(x):= (u^1(x), \ldots, u^m(x)), x \in \Omega,$ with $\Omega$ a bounded domain in $\R^n$ . In the scalar case, $m=1$, \eqref{VOP} is the Euler-Lagrange equation associated to the classical one-phase Bernoulli energy functional ($u \geq 0$),
\be\label{J}
J(u, \Omega):= \int_{\Omega} (|\nabla u|^2 + \chi_{\{u>0\}}) \ dx.
\ee

Minimizers of $J$ were first investigated systematically by Alt and Caffarelli. Two fundamental questions are answered in the pioneer article \cite{AC}, that is the Lipschitz regularity of minimizers and the regularity of ``flat" free boundaries, which in turns gives the almost-everywhere regularity of minimizing free boundaries. The viscosity approach to the associated free boundary problem was later developed by Caffarelli in \cite{C1,C3,C2}. In particular in \cite{C2} the regularity of flat free boundaries is obtained. There is a wide literature on this problem and the corresponding two-phase problem, and we refer the reader to the paper \cite{DFS} for a comprehensive survey.

The system \eqref{VOP} can also be seen as the Euler-Lagrange equations associated to a vectorial Alt-Caffarelli type functional. Namely, given a regular open set $\Omega\subset \R^n$ and $\Phi=(\varphi^1,\ldots,\varphi^m) \in H^{1/2}(\partial \Omega, \R^m)$, one can consider the vectorial free boundary problem
\be\label{min}
\text{min}\left\{\int_{\Omega}{\abs{\nabla U}^2\mathrm{d}x} + \abs{\Omega(U)}\colon U \in H^1(\Omega,\R^m),\, U=\Phi \text{ on }\partial \Omega \right\}.
\ee

In  \cite{
CSY}, the authors initiated the study of this problem where several flows
are involved, and interact whenever there is a phase transition. In particular, they applied a reduction method to reduce the problem to its scalar counterpart by assuming nonnegativity of the components of $U$. More precisely, under this assumption, the components are weak solutions of
$$
\Delta u^i = w^i \mathcal{H}^{n-1}\mres(\Omega \cap \partial^*\{\abs{U}>0\}), \quad\mbox{for }i=1,\dots,m,
$$
with
$$
w^i(x)= \lim_{y \in \{\abs{U}>0\},y\to x} \frac{u^i(y)}{\abs{U}(y)}.
$$

Recently in \cite{MTV2}, a different group of authors removed the sign assumption on the components. As expected, in this case the structure of the singular
set changes and the set of branching points $\mathrm{Sing}_2(F(U))$ naturally arises. More precisely, the following theorem holds.
\begin{thm}[\cite{MTV2}]
The problem \eqref{min} admits a solution $U \in H^1(\Omega;\R^m)$. Moreover, any solution is Lipschitz continuous in $\Omega \subset \R^n$ and the set $\Omega(U)$ has a locally finite perimeter in $\Omega$. More precisely, the free boundary $F(U)$ is a disjoint union of a regular part $\mathrm{Reg}(F(U))$, a (one-phase) singular set $\mathrm{Sing}_1(F(U))$ and a set of branching
points $\mathrm{Sing}_2(F(U))$:
\begin{enumerate}
  \item[1.] $\mathrm{Reg}(F(U))$ is an open subset of $F(U)$ and is locally the graph of a smooth function.
\item[2.] $\mathrm{Sing}_1(F(U))$ consists only of points in which the Lebesgue density of $\Omega(U)$ is strictly between $1/2$ and $1$. Moreover, there is $n^* \in \{5, 6, 7\}$ such that:
    \begin{itemize}
      \item if $n < n^*$, then $\mathrm{Sing}_1(F(U))$ is empty;
\item if $n = n^*$, then $\mathrm{Sing}_1(F(U))$ contains at most a finite number of isolated points;
\item if $n > n^*$, then the $(n - n^*)$-dimensional Hausdorff measure of $\mathrm{Sing}_1(F(U))$ is locally finite in $\Omega$.
    \end{itemize}
\item[3.] $\mathrm{Sing}_2(F(U))$ is a closed set of locally finite $(n-1)$-Hausdorff measure in $\Omega$ and consists of points in which the Lebesgue density of $\Omega(U)$ is $1$ and the blow-up limits are linear
functions.
\end{enumerate}
\end{thm}

As pointed out in \cite{MTV1}, problem \eqref{min} is also related to a class of shape optimization problems involving the eigenvalues of the Dirichlet Laplacian. Precisely, if $U_*$ is the vector whose components are the Dirichlet eigenfunctions on the set $\Omega_*$ which solves the shape optimization problem
\be\label{spectral}
\mbox{min}\left\{\sum_{i=1}^m \lambda_i(\Omega)\colon \Omega \subset \R^n \mbox{ open}, \abs{\Omega}=1\right\},
\ee
then $U_*$ can be seen as quasi-minimizers of \eqref{min}. Indeed, \cite{MTV2} follows some of the main ideas developed in \cite{MTV1}. In \cite{KL2, KL1}, a different set of authors considered an even more general class of spectral functionals than \eqref{spectral} and used a viscosity approach based on an Harnack inequality and a linearization, in the same spirit of the method developed in \cite{D} by the first author.

In this note, we are also inspired by \cite{D}, and we use a vectorial viscosity approach which does not reduce the problem to the scalar one-phase problem, as done in \cite{MTV2}. Since we work directly on the problem \eqref{VOP}, our proof (in particular the choice of barriers) is more  straightforward then the one in \cite{KL2,KL1} as it takes advantage of the fact that the norm $\abs{U}$ is a viscosity subsolution to the scalar one-phase problem.\\

One of the objectives of this note is to develop a method suitable for other vectorial problems, for example Bernoulli-type problems involving nonlocal diffusion. In particular, in \cite{CRS, DR, DSS} the authors studied the regularity of a one-phase scalar free boundary problem for the fractional Laplacian. While in \cite{CRS} general properties like optimal regularity, nondegeneracy and classification of global solutions were proved, in \cite{DR,DSS} the authors developed a viscosity approach in order to prove that flat free boundaries are actually $C^{1,\alpha}$. In a forthcoming paper, we plan to extend these results to the vectorial case, following the approach developed in this paper.\\

We now state our main theorem. From now on, we denote by $\{e_i\}_{i=1,\ldots,n}$,$\{f^i\}_{i=1,\ldots,m}$ canonical basis in $\R^n$ and $\R^m$ respectively. Unit directions in $\R^n$ and $\R^m$ will be typically denoted by $e$ and $f$. The Euclidean norm in either space is denoted by $|\cdot|,$ while the dot product is denoted by $\langle \cdot, \cdot \rangle.$

\begin{defn}\label{solutionnew} We say that $U \in C(\Omega, \R^m)$ is a viscosity solution to \eqref{VOP} in $\Omega$ if
$$\Delta u^i=0 \quad \text{in $\Omega(U)$}, \quad \forall i=1,\ldots,m;$$
and the free boundary condition is satisfied in the following sense. Given $x_0 \in F(U)$, and a test function $\varphi \in C^2$ in a neighborhood of $x_0$, with $|\nabla \varphi|(x_0) \neq 0,$ then
\begin{enumerate}
\item If $|\nabla \varphi|(x_0) >1$, then for all unit directions $f$ in $\R^m$, $\langle U, f\rangle $ cannot be touched by below by $\varphi$ at $x_0.$
\item If $|\nabla \varphi|(x_0) <1,$ then $\abs{U}$ cannot be touched by above by $\varphi$ at $x_0.$
\end{enumerate}
\end{defn}

Our main theorem reads as follows.

\begin{thm}\label{main} Let $U$ be a viscosity solution to \eqref{VOP} in $B_1$. There exists a universal constant $\bar \eps >0$ such that if $U$ is $\bar \eps$ flat in $B_1$, i.e. for some unit directions $e \in \R^n, f\in \R^m$
\be \label{flat} |U(x) - f \langle x , e \rangle ^+| \leq \bar \eps, \quad \text{in $B_1$,}
\ee
and \be\label{nondegenra}
 |U| \equiv 0 \quad \text{in $B_1 \cap \{\langle x , e \rangle  < - \bar \eps\}$}\ee
then $F(U) \in C^{1,\alpha}$ in $B_{1/2}.$
\end{thm}

We remark that condition \eqref{nondegenra} is satisfied by flat minimizing solutions in view of non-degeneracy \cite[Section 2.1]{MTV2}.

Notice that in \cite{MTV2} the authors used a smaller class of viscosity solutions in which property $(i)$ is replaced by the following:

$$\text{(i') \ If $|\nabla \varphi|(x_0) >1$, then $\abs{U}$ cannot be touched by below by $\varphi$ at $x_0$.}$$

Indeed, in \cite[Lemma 3.2.]{MTV2} they proved that if $U$ is a minimizing free boundary, then for every $x_0 \in \mathrm{Reg}(F(U))\cup \mathrm{Sing}_1(F(U))$ there exists a small radius $r>0$ such that $U$ is a viscosity solution of
\be \begin{cases}
\Delta U=0 & \text{in $\Omega(U)\cap B_r(x_0)$;}\\
 |\nabla |U||=1 & \text{on $F(U)\cap B_r(x_0),$}
\end{cases}\ee
in the sense of (i)'-(ii).  The larger class in Definition \ref{solutionnew} is better suited for the strategy of our proof,
which relies on a vectorial Harnack inequality and improvement of flatness technique. Details of the Harnack inequality are carried on in Section 2, while the improvement of flatness argument is presented in Section 3.\\

\section{Harnack type inequality}
In this Section we will prove a Harnack type inequality for solutions to problem \eqref{VOP}. Precisely, the following is our main theorem.

\begin{thm}\label{holder}
There exists a universal constant $\overline{\eps}>0$ such that, if $U$ solves \eqref{VOP} in $B_1$, and for some point $x_0 \in B^+_1(U) \cup F(U),$
\be\label{flat_2} x_n +a_0 \leq u^1 \leq |U| \leq (x_n+b_0)^+ \quad \text{in $B_r(x_0) \subset B_1$,}
\ee  with $$b_0 -a_0 \leq  \bar\eps r,$$
and
$$|u^i|\leq r\left( \frac{b_0-a_0}{r}\right)^{3/4} \quad \text{in $B_{1/2}(x_0)$}, \quad i=2,\ldots, m,$$ then
\be\label{flat_2} x_n +a_1 \leq u^1 \leq |U| \leq (x_n+b_1)^+ \quad \text{ in $B_{r/20}(x_0)$,}
\ee
with $$a_0 \leq a_1 \leq b_1 \leq b_0, \quad  b_1 - a_1= (1-c)(b_0-a_0) ,$$ for $0<c<1$ universal.
 \end{thm}

 We briefly postpone the proof of Theorem \ref{holder}, and obtain the key corollary which will be used in the improvement of flatness argument.
First, the following lemma allows to translate the flatness assumption
on the vector-valued function $U$ into the property that one of its components is trapped between nearby translation of a one-plane solution, while the remaining ones are small.

\begin{lem}\label{change}Let $U$ be a solution to  \eqref{VOP} in $B_1$ such that for $\eps>0$
\be\label{flat_harnack1} |U- f^1 x_n^+| \leq \eps, \quad \text{in $B_1$,}
\ee
and
\be\label{non_d11} |U| \equiv 0 \quad \text{in $B_1 \cap \{x_n < - \eps\}$}. \ee Then
\begin{enumerate}
\item For $i=2,\ldots,m, $ \be |u^i| \leq C \eps(x_n+\eps)^+ \quad \text{in $B_{3/4};$}\ee
\item \be (x_n-\eps) \leq u^1 \leq |U| \leq (x_n+2\eps)^+ \quad \text{in $B_1$.}\ee
\end{enumerate}
\end{lem}
\begin{proof} The bounds in $(ii)$ are an immediate consequence of the assumptions. For $(i), $let $v$ be the harmonic function in $B_1 \cap \{x_n >-\eps\}$ with smooth boundary data $\bar v$, $0 \leq \bar v \leq 1$ such that
$$
\begin{cases}
\bar v = 0 &\text{on $B_{1-\eps} \cap \{x_n=-\eps\}$};\\
\bar v=1  &\text{on $\p B_1 \cap \{x_n >-\eps\}.$}
\end{cases}
$$
Since $|u^i|$ is subharmonic and by \eqref{flat_harnack1}-\eqref{non_d11} $$|u^i| \leq \eps, \quad u^i\equiv 0 \quad \text{on $\{x_n =-\eps\}$},$$ by comparison and boundary regularity we get
$$|u^i| \leq \eps v \leq C \eps (x_n+\eps) \quad \text{in $B_{1/2} \cap \{x_n >-\eps\}$.}$$
\end{proof}

Now denote by,
$$\tilde u^1:= \frac{u^1-x_{n}}{\eps}, \quad \widetilde{|U|} : =\frac{|U|-x_n}{\eps}, \quad x \in B_1(U) \cup F(U).$$ The following corollary is a consequence of the results above.

\begin{cor}\label{AA} Let $U$ be a solution to  \eqref{VOP} in $B_1$ such that for $\eps>0$
\be\label{flat*} |U- f^1 x_n^+| \leq \eps, \quad \text{in $B_1$,}
\ee
and
\be\label{non_d*} |U| \equiv 0 \quad \text{in $B_1 \cap \{x_n < - \eps\}$}. \ee There exists $\bar \eps>0$ small universal, such that if $\eps \leq \bar \eps$, then $\tilde u^1$ and $\widetilde{|U|}$ have a universal H\"older modulus of continuity at $x_0 \in B_{1/2}$ outside a ball of radius $r_\eps,$ with $r_\eps \to 0$ as $\eps \to 0.$
\end{cor}
\begin{proof} In view of Lemma \ref{change}, $u^1, U$ satisfy the assumptions of Theorem \ref{holder} (for  $\bar \eps$ possibly smaller than the one in Theorem \ref{holder}), with $$a_0 = -\eps, \quad b_0= 2\eps, \quad r= 1/4, \quad x_0 \in B_{1/2}.$$ Hence,
\be x_n +a_k \leq u^1 \leq |U| \leq (x_n+b_k)^+ \quad \text{in $B_{r\rho_k}(x_0)$}, \quad \rho_k=20^{-k},
\ee
with $$b_k - a_k =(1-c)^k(b_0-a_0),$$ for $0<c<1$ universal
as long as,
$$20^k (1-c)^k \eps \leq \bar \eps, \quad \eps \leq \bar C \left(\frac{(1-c)^{3}}{20}\right)^k,$$ 
with $\bar C$ universal. This implies that for such cases, in $(\Omega(U)\cup F(U))\cap B_{r\rho_k}(x_0)$ the oscillation of the functions $\tilde u^1$ and $\widetilde{|U|}$ are less or equal than $(1 - c)^k = 20^{-\alpha k} = \rho^\alpha$, as we claimed.
\end{proof}

The next lemma is the main ingredient in the proof of Theorem \ref{holder}. It uses the observation that $|U|$ is subharmonic in $\Omega(U)$, as it can be easily verified with a straightforward computation.

\begin{lem}\label{fbharnack} Let $U$ be a solution to  \eqref{VOP} in $B_1$ such that for $\eps>0$
\be\label{flat_harnack2} p(x) \leq u^1 \leq |U| \leq (p(x) + \eps)^+ \quad \text{in $B_1$,} \quad p(x):= x_n + \sigma, |\sigma| <1/10,
\ee and
\be\label{uismall}|u^i|\leq \eps^{3/4} \quad \text{in $B_{1/2}$}, \quad i=2,\ldots, m,\ee
with $C>0$ universal.
There exists $\bar \eps>0$, such that if $0 < \eps \leq \bar \eps$, then at least one of the following holds true:
\be\label{flat_harnack3} p(x) +c\eps \leq u^1 \leq |U| \quad \text{in $B_{1/2}$,}
\ee
or
$$u^1 \leq |U| \leq (p(x) + (1-c)\eps)^+,  \quad \text{in $B_{1/2}$,}
$$
for $0<c<0$ small universal.
\end{lem}
\begin{proof}
We distinguish two cases. If at $\bar x = \frac 1 5 e_n$
\be\label{bottom}u^1(\bar x) \leq p(\bar x) + \frac{\eps}{2}, \ee
then we will show that
\be\label{double}|U| \leq (p(x) + (1-c) \eps)^+ \quad \text{in $B_{1/2}$ }.\ee
Similarly, if
\be \label{secondcase} u^1(\bar x) \geq p(\bar x)+ \frac{\eps}{2},\ee we will show that $$u^1 \geq  p(x)+ c \eps \quad \text{in $B_{1/2}$}.$$
In either case, we let $A = B_{3/4}(\bar{x})\setminus B_{1/20}(\bar{x})$ and
\be\label{w}
w=
\begin{cases}
1 & \text{in $B_{1/20}(\bar{x})$};\\
\ddfrac{\abs{x-\bar{x}}^{\gamma} - (3/4)^{\gamma}}{(1/20)^{\gamma} - (3/4)^{\gamma}} &\text{in $A$};\\
0 &\text{on $\partial B_{3/4}(\bar{x})$}.
\end{cases}
\ee
with $\gamma <0$ be such that $\Delta w >0$ in $A$.\\

{\it Case 1.} If $u^1(\bar x) \geq p(\bar x) + \frac{\eps}{2},$ the argument in \cite[Lemma 3.3]{D} carries on, even if $u^1$ may change sign. For completeness, we provide the details.

Since $\abs{\sigma}<1/10$ and  by the flatness assumption
\be\label{flat} u^1 - p\geq 0 \quad \text{in $B_1$},\ee
we immediately deduce that $B_{1/10}(\bar x) \subset B_1^+(U)$. Notice that, by definition of $\bar{x}$, we have
\be\label{inclus}
B_{1/2} \subset\subset B_{3/4}(\overline{x}) \subset \subset B_1.
\ee
Hence, in view of \eqref{flat}, by Harnack inequality applied to $u^1-p$ we get for $c_0>0$ universal
\be\label{first_step}
u^1(x) -p(x) \geq c(u^1(\overline{x}) - p(\overline{x})) \geq c_0 \eps \quad \text{in $B_{1/20}(\bar x)$},
\ee
where in the second inequality we used assumption \eqref{secondcase}.\\
Now, let us set
\be\label{def.sub}
v_t(x) = p(x) + c_0 \eps (w(x)-1) +t \quad \text{in $\overline{B}_{3/4}(\bar{x})$}
\ee
for $t\geq 0$. Thus, we deduce that $\Delta v_t = \Delta p + c_0\eps \Delta w >0$ on $A$ and, by \eqref{flat}, we get $$v_0\leq p \leq u^1 \quad\text{in $\overline{B}_{3/4}(\overline{x})$}.$$
Thus, let $\overline{t}>0$ be the largest $t>0$ such that $v_t\leq u^1$ in $\overline{B}_{3/4}(\overline{x})$. We want to show that $\overline{t}\geq c_0\eps$. Indeed, by the definition of $v_t$, we will get
$$
u^1 \geq v_{\overline{t}} \geq p + c_0 \eps w \quad\text{in $\overline{B}_{3/4}(\overline{x})$}.
$$
In particular, by \eqref{inclus}, since $w \geq c_1$ on $\overline{B_{1/2}}$, we get
$$
u^1 \geq p + c\eps  \quad\text{in $\overline{B}_{1/2}$},
$$
as we claimed.

 Suppose by contradiction that $\bar{t} < c_0 \eps$. Let $\tilde{x} \in \overline{B}_{3/4}(\bar{x})$ be the touching point between $v_{\overline{t}}$ and $u$, i.e.
$$
u(\tilde{x}) = v_{\overline{t}}(\tilde{x}),
$$
we want to prove that it can only occur on $\overline{B}_{1/20}(\overline{x})$. Since $w \equiv 0$ on $\partial B_{3/4}(\overline{x})$ and $\overline{t}< c_0\eps$ we get
$$
v_{\overline{t}} = p - c_0\eps +\overline{t} < u \quad \text{on $\partial B_{3/4}(\overline{x})$},
$$
thus we left to exclude that $\tilde{x}$ belongs to the annulus $A$. By the definition \eqref{def.sub}, we get
\be
\abs{\nabla v_{\overline{t}}} \geq \abs{v_n} \geq \abs{1+c_0\eps w_n} \quad \text{in $A$}.
\ee
Since $w$ is radially symmetric $w_n = \abs{\nabla w} \nu_x \cdot \mathrm{e}_n$ in $A$, where $\nu_x$ is the unit direction of $x-\overline{x}$. On one side, from the definition of $w$, we get that $\abs{\nabla w} >c $ on $A$ and on the other $\nu_x \cdot \mathrm{e}_n$ is bounded by below in the region $\{v_{\overline{t}} \leq 0\}\cap A$, since $\overline{x}_n = 1/5$ and for $\eps$ small,
$$
\{v_{\overline{t}}\leq 0\} \cap A \subset \{p - c_0\eps \leq 0\} \cap A = \{x_n \leq - \sigma + c_0\eps \}\cap A \subset \{x_n < 3/20 \}.
$$
Hence, we infer that $\abs{\nabla v_{\overline{t}}} \geq 1+ c_2(\gamma)\eps$ in $\{v_{\overline{t}}\leq 0 \}\cap A$
and consequently
\be\label{>1}
  \abs{\nabla v_{\overline{t}}} >1 \quad\text{on $ F(v_{\overline{t}})\cap A$.}
\ee
Finally, since we observed that $\Delta v_{\overline{t}} > \eps^2$ in $A$, and $v_{\bar t } \leq u^1$, we deduce that the touching cannot occur in $A \cap B_1^+(U)$ where $u^1$ is harmonic. In view of \eqref{>1} and Definition \ref{solutionnew}, we conclude that the touching cannot occur on $A \cap F(U)$ as well.
 Therefore $\tilde{x} \in \overline{B}_{1/20}(\overline{x})$ and
$$
u^1(\tilde{x}) = v_{\overline{t}}(\tilde{x}) = p(\tilde{x})  + \overline{t} \leq p(\tilde{x}) + c_0 \eps,
$$
in contradiction with \eqref{first_step}.\\

{\it Case 2.} If $u^1(\bar x) \leq p(\bar x) + \frac{\eps}{2},$ by the lower bound in \eqref{flat_harnack2}, $|u^1|=u^1$ in $B_{1/10}(\bar x) \subset B_1^+(U)$.
Thus by Harnack inequality and assumption \eqref{bottom}
\be \label{c_0}
p +\eps - |u^1| \geq  2c_0 \eps \quad \text{in $B_{1/20}(\bar x)$.}
\ee

Since the desired bound clearly holds in $\{p \leq -\eps \}$, where all the $u^i \equiv 0,$ it is enough to restrict to the region $\{p> -\eps\}$. Below, the superscript $\eps$ denotes such restriction.

Now, let us consider for $t\geq 0$
\be\label{def.sub2}
v_t(x) = p(x)+\eps - c_0 \eps (w(x)-1) -t \quad \text{in $\overline{B^\eps}_{3/4}(\bar{x})$.}
\ee
Thus, we have $\Delta v_t = -c_0 \eps \Delta w< 0$ on $A^\eps$ and
$$
v_0 \geq p +\eps \geq \abs{U} \quad \text{in $\overline{B^\eps}_{3/4}(\overline{x})$}.
$$
Now, let $\bar{t}>0$ be the largest $t>0$ such that $\abs{U}\leq v_t$ in $\overline{B^\eps}_{3/4}(\overline{x})$.
We want to show that $\overline{t}\geq c_0\eps$. Indeed, by the definition of $v_t$, this would give
$$
\abs{U} \leq v_{\bar{t}} \leq p +\eps - c_0 \eps w \quad\text{in $\overline{B^\eps}_{3/4}(\overline{x})$}.
$$
In particular, by \eqref{inclus}, since $w \geq c_1$ on $\overline{B}_{1/2 }$, we get
$$
\abs{U} \leq p+ (1-c)\eps  \quad\text{in $\overline{B^\eps}_{1/2}$},
$$
as we claimed.

We are left with the proof that $\overline{t}\geq c_0\eps$. Suppose by contradiction that $\bar{t} < c_0 \eps$. Let $\tilde{x} \in \overline{B^\eps}_{3/4}(\bar{x})$ be the first touching point between $v_{\bar{t}}$ and $\abs{U}$ in $\overline{B^\eps}_{3/4}(\bar{x})$, i.e.
$$
\abs{U}(\tilde{x}) = v_{\bar{t}}(\tilde{x}).
$$
We prove that such touching point  can only occur on $\overline{B}_{1/20}(\overline{x})$. Since $w \equiv 0$ on $\partial B_{3/4}(\overline{x})$, $|U| \equiv 0$ on $\{p=-\eps\}$ and $\bar{t}< c_0\eps$ we get
$$
v_{\bar{t}} 
> \abs{U} \quad \text{on $\partial B^\eps_{3/4}(\overline{x})$},
$$
thus we need to exclude that $\tilde{x}$ belongs to $A^\eps$. By the definition \eqref{def.sub}, we get
\be
\abs{\nabla v_{\bar{t}}}^2 = 1 - 2c_0 \eps w_n + O(\eps^2) \quad \text{in $A^\eps$}.
\ee On the other hand, it easily follows from the definition \eqref{def.sub2} that $$\{v_{\bar t}=0\} \subset \{ p +(1-c_0)\eps <0\}\subset \left\{x_n <\frac{1}{10} - (1-c_0)\eps\right\},$$ thus we can estimate that
$$w_n \geq c_3(\gamma)>0 \quad \text{on $A^\eps \cap \{v_{\bar t}=0\}.$}$$

 Hence, we infer that for $\eps$ small, \be\label{<1}
 0 \neq  \abs{\nabla v_{\bar{t}}} <1 \quad\text{on $ F(v_{\bar{t}})\cap A^\eps$.}
\ee
Finally, since we observed that $\Delta v_{\bar{t}}< 0$ in $A^\eps$, and $v_{\bar t} \geq |U|$, we deduce that $\tilde x \not \in A^\eps \cap B^1(U)$. Moreover, by \eqref{<1} and Definition \ref{solutionnew}, we also conclude that $\tilde x \not \in A^\eps \cap F(U).$

Therefore, $\tilde x \in \overline{B}_{1/20}(\bar x)$ and
$$|U|(\tilde x)= v_{\bar t}(\tilde x)= p(\tilde x) + \eps - \bar t > p(\tilde x) +\eps - c_0\eps,$$
that is
$$p(\tilde x) +\eps- |U|(\tilde x) < c_0 \eps.$$ This implies, using \eqref{uismall} and the fact that $|u^1|$ is bounded,
$$p(\tilde x) +\eps - |u^1|(\tilde x) - C\eps^{3/2} < c_0 \eps,$$
and we contradict \eqref{c_0}, for $\eps$ small and $C$ universal constant.
\end{proof}

We are now ready to prove Theorem \ref{holder}.
\begin{proof} Let us rescale,
$$u^i_r(x):=\frac{1}{r} u^i(r x+x_0), \quad x \in B_1, \quad \quad i=1, \ldots, m.$$ Then,
\be p(x) \leq u_r^{1} \leq |U_r| \leq (p(x)+\eps)^+ \quad \text{in $B_{1}$,}
\ee with
$$\eps:= r^{-1}(b_0 -a_0) \leq \bar \eps, \quad p(x)=x_n +\sigma, \quad \sigma = r^{-1}a_0$$
and
\be\label{trequarti}|u^i_r| \leq \eps^{3/4} \quad \text{in $B_{1/2}$}.\ee
If $$|a_0| \leq \frac{r}{10},$$ then we can apply Lemma \ref{fbharnack} and reach the desired conclusion. If $a_0 < -r/10$ then for $\eps$ small,
$$|U_r| \equiv 0 \quad \text{in $B_{1/20}$},$$ and again we obtain the claim. We are left with the case $a_0 > r/10.$ Then $B_{1/10} \subset B_1^+(U_r)$ and $u^1_r>0$ and harmonic in $B_{1/10}$. Hence by standard Harnack inequality, either
$$u^1_r \geq p(x) + c \eps \quad \text{in $B_{1/20}$},$$ and we are done, or
$$u^1_r \leq p(x) + (1-c) \eps \quad \text{in $B_{1/20}$}.$$
Finally, by \eqref{trequarti}, for $\eps$ sufficiently small $$|U_r| \leq u_r^1 + C \eps^{{3/2}}\leq  p(x) + \left(1-\frac{c}{2}\right) \eps \quad \text{in $B_{1/20}$},$$
with $C,c>0$ universal.
\end{proof}

\section{The Improvement of flatness}

In this section we prove our main result, an improvement of flatness lemma, from which the desired Theorem \ref{main} follows by standard techniques (see for example \cite{CSbook}.)

First, we recall some known facts.
Consider the following boundary value problem, which is the linearized problem arising from our improvement of flatness technique:
\be\begin{cases}\label{linearized}
 \Delta \widetilde U=0 & \text{in $B_{1/2} \cap \{x_n>0\},$}\\
\frac{\p}{\p x_n}(\tilde u^1)=0, \quad \tilde  u^i=0 \quad i=2, \ldots, m &\text{on $B_{1/2} \cap \{x_n=0\}$},
\end{cases}\ee
with $\widetilde{U}=(\tilde u^1, \ldots, \tilde u^m) \in C(B_{1/2}\cap \{x_n \geq 0\}, \R^m)$.
The Neumann problem for $\tilde u^1$ is satisfied in the following viscosity sense.

\begin{defn}\label{defN} If $P(x)$ is a quadratic polynomial touching
$\tilde u^1$ by below (resp. above) at $\bar x \in B_{1/2} \cap \{x_n
\geq 0\}$, then

\

(i) if $\bar x \in B_{1/2} \cap \{x_n >0\}$ then $\Delta P \leq 0,$
(resp. $\Delta P \geq 0$) i.e $\tilde u^1$ is harmonic in the
viscosity sense;

\

(ii) if $\bar x \in B_{1/2} \cap \{x_n=0\}$ then $P_n(\bar x) \leq
0$ (resp. $P_n(\bar x) \geq 0$.)

\end{defn}

As usual, in the definition above we can
choose polynomials $P$ that touch $\tilde u^1$ strictly by
below/above . Also, it suffices to verify that (ii) holds for polynomials
$\tilde P$ with $\Delta \tilde P
> 0$. \\

Since the linearized problem \eqref{linearized} is a system completely decoupled, the regularity of solutions follows immediately by standard theory (see also Lemma 2.6 in \cite{D}.)

\begin{lem}\label{regularity}
Let $\widetilde U$ be a viscosity solution to \eqref{linearized} in $\Omega$. Then $\widetilde{U}$ is a classical solution to \eqref{linearized} and $\widetilde{U} \in C^\infty(B_{1/2} \cap \{x_n\geq 0\};\R^m)$.
\end{lem}

We are now ready to state and prove our key lemma.

\begin{lem}[Improvement of Flatness] Let $U$ be a viscosity solution to \eqref{VOP} in $B_1$ satisfying the $\eps$-flatness assumption in $B_1$
\be\label{flat1}  |U- f^1 x_n^+| \leq \eps \quad \text{in $B_1$},
\ee
and
\be\label{non_d1} |U| \equiv 0 \quad \text{in $B_1 \cap \{x_n < - \eps\}$},\ee
with $0 \in F(U).$ If $0<r \leq r_0$ for a universal $r_0>0$, and $0<\eps \leq \eps_0$ for some $\eps_0$ depending on $r$, then
\be\label{flat_imp}
|U - \bar f \langle x, \nu\rangle^+| \leq \eps \frac r 2\quad \text{in $B_r$},
\ee
and
\be\label{non_dimpr} |U| \equiv 0 \quad \text{in $B_r \cap \left\{\langle x , \nu \rangle < - \eps \frac{r}{2}\right\}$}, \ee
with $|\nu -e_n|\leq C\eps, |\bar f-f^1| \leq C \eps$, for a universal constant $C>0.$
\end{lem}
\begin{proof}
Following the strategy of [D], we divide the proof in three different steps.\\

{\it Step 1 - Compactness.} Fixed $r \leq r_0$ with $r_0$ universal (the value of $r_0$ will be given in Step 3), suppose by contradiction that there exists $\eps_k \to 0$ and a sequence of solutions $(U_k)_k$ of \eqref{VOP} such that $0 \in F(U_k)$ and \eqref{flat1} and \eqref{non_d1} are satisfied for every $k$, i.e.
\be\label{contrad1}
|U_k- f^1 x_n^+| \leq \eps_k, \quad \text{in $B_1$,}
\ee
and
\be\label{contrad2}
|U_k| \equiv 0 \quad \text{in $B_1 \cap \{x_n < - \eps_k\}$},\ee
but  the conclusions \eqref{flat_imp} and \eqref{non_dimpr} of the Lemma do not hold.\\
Let us set
\be\label{tilde}
\widetilde U_k = \frac{U_k - f^1 x_n}{\eps_k}, \quad V_k = \frac{|U_k|-x_n}{\eps_k} \quad \text{in $\Omega(U_k):=B_1(U_k) \cup F(U_k) \subset \{x_n  \geq -\eps_k\}$.}
\ee
By the flatness assumptions \eqref{contrad1}-\eqref{contrad2}, $(U_k)_k$ and $(V_k)_k$ are uniformly bounded in $B_1$. Moreover, $F(U_k)$ converges to $B_1 \cap \{x_n = 0\}$ in the Hausdorff distance. Now, by Corollary \ref{AA} and Ascoli-Arzela, it follows that, up to a subsequence, the graphs of the components $\widetilde{u}^i_k$ of $\widetilde{U}_k$ and of $V_k$ over $B_{1/2} \cap (B_1(U_k) \cup F(U_k))$ converge in the Hausdorff distance to the graph of Holder continuous functions $\tilde u^i_\infty, V_\infty$ on $B_{1/2} \cap \{x_n \geq 0\}$, for every $i=1,\ldots,m$. Moreover, by Corollary \ref{AA}, 
\be\label{equal}V_\infty \equiv \tilde{u}_\infty^1 \quad \text{in $B_{1/2} \cap \{x_n \geq 0\}.$}\ee
{\it Step 2 - Linearized problem.} We show that $\widetilde U_\infty$ satisfies the following problem in the viscosity sense:
\be\begin{cases}
 \Delta \widetilde U_\infty =0 & \text{in $B_{1/2} \cap \{x_n>0\},$}\\
\frac{\p}{\p x_n}(\tilde u^1_\infty)=0, \quad \tilde  u^i_\infty=0 \quad i=2, \ldots, m &\text{on $B_{1/2} \cap \{x_n=0\}.$}
\end{cases}\ee
In view of Lemma \ref{change}, part $(i)$, the conclusion for $i=2,\ldots, m$ is immediate. We are left with the case $i=1.$

First, let us consider the case a polynomial $P$ touches $\tilde{u}^1$ at $\overline{x} \in B_{1/2} \cap \{ x_n \geq 0\}$ strictly by below. Then the arguments of \cite{D} apply. Indeed, we need to show that
\begin{enumerate}
\item if $P$ touches $\tilde u^1_\infty$ at $\overline{x} \in B_{1/2}\cap \{x_n>0\}$, then $\Delta P(\overline{x}) \leq 0$,
\item if $P$ touches $\tilde u^1_\infty$ on $\{x_n=0\}$, then $P_n(\overline{x})\leq 0$,
\end{enumerate}
Since $\tilde u^1_k \to \tilde u^1_\infty$ uniformly on compacts, there exists $(x_k)_k \subset B_{1/2} \cap (B_1(U_k) \cup F(U_k))$, with $x_k \to \overline{x}$, and $c_k \to 0$ such that
$$
\tilde{u}^1_k(x_k) = P(x_k)+c_k,
$$
and $\tilde{u}^1_k \geq P+c_k$ in a neighborhood of $x_k$. From the definition of the sequence $\tilde u^1_k$, we infer $u^1_k(x_k)=Q(x_k)$ and $u^1\geq Q$ in a neighborhood of $x_k$, with
\be
Q(x)= x_n +\eps_k (P(x)+c_k)\ee
If $\overline{x}\in B_{1/2}\cap \{x_n >0\}$, then $x_k \in B_{1/2}(U_k)$, for $k$ sufficiently large, and hence since $Q$ touches $u^1_k$ by below at $x_k$
$$
\Delta Q(x_k) = \eps_k \Delta P(x_k) \leq 0,
$$
which leads to $\Delta P(\overline{x}) \leq 0$ as $k\to \infty$.\\
Instead, if $\overline{x} \in B_{1/2}\cap \{x_n =0\}$, then we can assume $\Delta P >0$. It is not restrictive to suppose that, for $k$ sufficiently large, $x_k \in F(U_k)$. Otherwise $x_{k_n} \in B_{1/2}(U_{k_n})$ for a subsequence $k_n \to \infty$ and in that case $\Delta P(x_{k_n}) \leq 0$, in contradiction with the strict subharmonicity of $P$.\\
Thus, for $k$ large,  $x_k \in F(U_k)$. Then noticed that $\nabla Q = e_n + \eps \nabla P$ and $\abs{\nabla Q} > 0$ for $k$ sufficiently large, since $Q$ touches $u^1_k$ by below, by Definition \ref{solutionnew} we deduce that $\abs{\nabla Q}^2(x_k)\leq 1$, i.e.
      $$
      \eps_k \abs{\nabla P}^2(x_k) + 2 P_n(x_k) \leq 0.
      $$

Passing to the limit as $k \to \infty$ we obtain the desired conclusion.\\

Consider now the case when $P$ touches $\tilde{u}^1_\infty$ at $\overline{x} \in B_{1/2} \cap \{ x_n \geq 0\}$ strictly by above. Since the case $\overline{x}\in B_{1/2} \cap \{x_n >0\}$ follows the same reasoning of the previous part, we move on to the case $\overline{x} \in B_{1/2} \cap \{x_n=0\}$ and assume that $\Delta P <0$. We claim that
$$P_n(\overline{x})\geq 0.$$

Since $\tilde u^1_\infty = V_\infty,$ for $k$ sufficiently large we get
 $
 \abs{U_k}(x_k) = Q(x_k)
 $
 and $
 \abs{U_k} \leq Q
 $
 in a neighborhood of $x_k \to \bar x$, with
 $$
 Q(x)=x_n + \eps_k (P - c_k), \quad c_k \to 0.
 $$ As before, since $|U_k|$ is subharmonic, we can assume that $x_k \in F(U_k).$
 By the definition of viscosity solution, we deduce that $\abs{\nabla Q}^2(x_k)\geq1$, i.e.
 $$
 \eps_k \abs{\nabla P}^2(x_k) + 2P_n(x_k) \geq 0,
 $$
 which leads to the claimed result as $k\to \infty$.
{\color{white}.}\\

{\it Step 3 - Improvement of flatness.} Since $(\widetilde{U}_k)_k$ is uniformly bounded in $B_1$, we get a uniform bound on $|\tilde u^i_\infty|$, for every $i=1,\ldots,m$.
Furthermore, since $0 \in F(\widetilde{U}_\infty)$, by the regularity result in Lemma \ref{regularity} we deduce that
$$
\abs{\tilde u^i_\infty (x) - \langle \nabla \tilde u^i_\infty(0),x\rangle}\leq C_0 r^2 \quad \text{in $B_r \cap \{x_n \geq 0\}$}, \quad i=1,\ldots, m
$$
for a universal constant $C_0>0$. On one side, since $\partial_{x_n}(\tilde u^1_\infty)=0$ on $B_{1/2}\cap \{x_n=0\}$, we infer
$$
\langle x', \tilde{\nu}^1 \rangle - C_0 r^2 \leq \tilde{u}^1_\infty(x) \leq \langle x', \tilde{\nu}^1\rangle + C_0 r^2 \quad \text{in $B_r \cap \{x_n\geq 0\}$},
$$
where $\tilde{\nu}^1 = \nabla \tilde{u}^1_\infty (0)$ is a vector in the variables $x_1,\ldots,x_{n-1}$, with $\abs{\tilde{\nu}^1}\leq M$, for some $M$ universal constant. Thus, for $k$ sufficiently large, there exists $C_1$ such that
$$
\langle x', \tilde{\nu}^1 \rangle - C_1 r^2 \leq \tilde{u}^1_k(x) \leq \langle x', \tilde{\nu}^1\rangle + C_1 r^2 \quad \text{in $\Omega(U_k) \cap B_r$}
$$
and exploiting the definition of $\tilde u^1_k$ we read
$$
x_n + \eps_k\langle x', \tilde{\nu}^1 \rangle - \eps_k C_1 r^2 \leq u^1_k(x) \leq x_n + \eps_k \langle x', \tilde{\nu}^1\rangle + \eps_k C_1 r^2 \quad \text{in $\Omega(U_k) \cap B_r.$}
$$
Thus, called
$$
\nu = \frac{(\eps_k \tilde{\nu}^1,1)}{\sqrt{1+\eps_k^2\abs{\tilde \nu^1}^2}} \in S^{n},
$$
since for $k$ sufficiently large $1\leq \sqrt{1+\abs{\tilde{\nu}^1}^2\eps_k^2} \leq 1+ M^2\eps_k^2/2$, we deduce that
\be\label{improv.1}
\langle x, \nu \rangle -\frac{r}{2}M^2\eps_k^2- \eps_k C_1 r^2 \leq u^1_k(x) \leq \langle x, \nu\rangle +\frac{r}{2}M^2\eps_k^2+ \eps_k C_1 r^2 \quad \text{in $\Omega(U_k) \cap B_r$}.
\ee
It follows that, for $r_0 \leq 1/(8C_1)$ and $k$ large,
\be\label{zero} F(U_k) \cap B_r \subset \left\{|\langle x, \nu \rangle| \leq \eps_k \frac{r}{4}\right\}\ee
and since $|U_k| \equiv 0$ in $\{x_n <-\eps_k\}$ and  $|U_k| >0$ in $\{x_n >\eps_k\}$, we conclude that
\be\label{zero+} |U_k| \equiv 0 \quad \text{in $B_r \cap \left\{\langle x, \nu\rangle  < - \eps_k \frac{r}{4}\right\}$},\ee
and
\be\label{pos} |U_k|>0 \quad \text{in $B_r \cap \left\{\langle x, \nu\rangle >\eps_k \frac r 4\right\}.$}\ee
Also, for $r_0 < 1/(8C_1)$ and $k$ large,
\be\label{improv.2}
\langle x, \nu \rangle -\frac{r}{8}\eps_k\leq u^1_k(x) \leq \langle x, \nu\rangle +\frac{r}{8}\eps_k\quad \text{in $\Omega(U_k) \cap B_r$}.
\ee

On the other, since $\tilde{u}^i_\infty = 0$ on $B_{1/2}\cap \{x_n=0\}$, for $i=2,\ldots,m$ we get
$$
\langle x, \tilde\nu^i\rangle - C_0 r^2 \leq \tilde{u}^i_\infty(x) \leq \langle x, \tilde\nu^i\rangle + C_0 r^2 \quad \text{in $B_r \cap \{x_n\geq 0\}$},
$$
where $\tilde\nu^i=M^ie_n, |M^i| \leq M$, for $M$ universal constant,  and for $k$ sufficiently large
\be\label{improv.i}
|u^i_k(x) - M^ix_n \eps_k| \leq\frac{r}{8}\eps_k \quad \text{in $\Omega(U_k) \cap B_r$}.
\ee
Finally, set
$$
\bar f^1_k=\frac{1}{\sqrt{1+\eps_k^2\sum_{i=2}^m \abs{M^i}^2}} \quad\text{and}\quad \bar f^i_k=\eps_k \frac{M^i}{\sqrt{1+\eps_k^2\sum_{i=2}^m \abs{M^i}^2}}
$$
for $i=2,\ldots,m$. Thus, by \eqref{improv.2} we get for $k$ large,
\be\label{ineq.1}
\abs{u^1_k - \bar f_k^1 \langle x, \nu \rangle } 
\leq \abs{u^1_k - \langle x,\nu \rangle }
+ \frac{(m-1)M^2}{2}\eps_k^2 r\leq \frac{r}{4}\eps_k, \quad \text{in $\Omega(U_k) \cap B_r$},
\ee
and similarly, by \eqref{improv.i}, we obtain for the other components
\begin{align}\label{ineq.i}
\begin{aligned}
\abs{u^i_k - \bar f_k^i \langle x, \nu\rangle }
& \leq \abs{u^1_k - \eps_k M^i x_n } + \eps_k M r \abs{e_n - f^1_k \nu}\\
& \leq \frac{r}{8}\eps_k + \frac m 2 \eps_k^3 M^3 r \leq \frac{r}{4}\eps_k, \quad \text{in $\Omega(U_k) \cap B_r$}.
\end{aligned}
\end{align}
Summing \eqref{ineq.1} and \eqref{ineq.i} for all $i=2,\cdots,m$ we finally get
$$
\abs{U_k - f_k \langle x, \nu\rangle }\leq \eps_k \frac{r}{4}\quad \text{in $\Omega(U_k) \cap B_r$.}
$$
In view of \eqref{zero+}, we only need to show that we can replace $\langle x, \nu\rangle$ with its positive part, in the region
$$-\frac{r}{4} \eps_k < \langle x, \nu \rangle <0.$$ Since in this region,
$$|U_k | \leq |U_k -\bar f_k\langle x, \nu\rangle| + |\bar f_k \langle x, \nu\rangle| < \eps_k \frac{r}{4},$$
we obtain,
$$|U_k -\bar f_k \langle x, \nu\rangle^+| \leq \eps_k \frac{r}{2},  \quad \text{in $\Omega(U_k) \cap B_r$}.$$ In view of \eqref{pos}, this inequality holds in $B_r$, which combined with \eqref{zero+} leads us to a contradiction.
\end{proof}

\bibliographystyle{plain}
\bibliography{bibliography}

\end{document}